\documentclass[10pt,a4paper,oneside]{article}
\usepackage[english]{babel}
\usepackage{a4wide,amsmath,amsthm,amssymb,url,graphicx,xspace,algorithm,algorithmic,pgf}
\usepackage[latin1]{inputenc}
\usepackage{enumitem}
\usepackage{multirow}
\usepackage{hyperref}
\usepackage{tabularx}
\usepackage{tikz}
\usepackage{allrunes}
\usepackage{bm}

\def\N{\mathbb{N}}

\def\L{\mathcal{L}}
\def\A{\mathcal{A}}

\DeclareMathOperator{\AG}{AG}

\theoremstyle{definition}
\newtheorem{theorem}{Theorem}[section]
\newtheorem{lemma}[theorem]{Lemma}
\newtheorem{definition}[theorem]{Definition}

\newtheorem{corollary}[theorem]{Corollary}

\newtheorem*{maintheorem}{Main Theorem}

\newcommand{\comments}[1]{}

\author{Maarten De Boeck \and Geertrui Van de Voorde}
\title{A new lower bound for the size of an affine blocking set}
\date{}
\begin{document}
\maketitle

\begin{abstract} A blocking set in an affine plane is a set of points $B$ such that every line contains at least one point of $B$. The best known lower bound for blocking sets in arbitrary (non-desarguesian) affine planes was derived in the 1980's by Bruen and Silverman \cite{bsil}. In this note, we improve on this result by showing that a blocking set of an affine plane of order $q$, $q\geq 25$, contains at least $q+\lfloor\sqrt{q}\rfloor+3$ points.

\end{abstract}


\section{Introduction and preliminaries}

\begin{definition}
	A blocking set in an affine plane $\A$ (or an {\em affine blocking set}) is a point set $B$ such that every line of $\A$ contains at least one point of $B$.
\end{definition}

In general, little is known about the smallest affine blocking sets, especially when compared to the knowledge about small blocking sets in projective planes. For the classical (desarguesian) affine plane $\AG(2,q)$ we have the following result, obtained independently by Jamison and by Brouwer and Schrijver.

\begin{theorem}[\cite{bs} and \cite{jam}]\label{jamison}
	A blocking set of $\AG(2,q)$ has size at least $2q-1$.
\end{theorem}

Note that this bound is also valid for affine blocking sets (so, in arbitrary affine planes) containing a line. The bound from Theorem \ref{jamison} cannot be generalised to arbitrary affine planes; in \cite{db}, it is shown that there are affine planes of order $n$ admitting blocking sets of size at most $\frac{4}{3}n+\frac{5}{3}\sqrt{n}$.

For arbitrary affine planes we have the following result by Bruen and Thas.

\begin{theorem}[{\cite[Cor. 2]{bt}}]
	If $\mathcal{S}$ is a blocking set of an axiomatic affine plane of order $q$, then
	\[
		|\mathcal{S}|\geq q+1+\frac{\sqrt{4q^{3}-4q+1}-1}{2q}\;.
	\]
	If $q$ is a square, this implies $|\mathcal{S}|\geq q+\sqrt{q}+1$.
\end{theorem}

This result was improved by Bruen and Silverman.

\begin{theorem}[{\cite[Theorem 3.1]{bsil}}]\label{silverman}
	If $\mathcal{S}$ is a blocking set of an axiomatic affine plane of order $q>3$, 
	then
	\[
	|\mathcal{S}|>q+\sqrt{q}+1\;.
	\]
	If $q\geq4$ is a square, this implies $|\mathcal{S}|\geq q+\sqrt{q}+2$.
\end{theorem}

We can embed an affine plane $\mathcal{A}$ in a projective plane $\mathcal{P}$ by adding the parallel classes as points and a unique line $\ell_{\infty}$ containing precisely these points (the line `at infinity'). This process is called \emph{completion}. By adding one point on $\ell_{\infty}$ to a blocking set of $\mathcal{A}$ we get a blocking set of $\mathcal{P}$. Using this correspondence, we can see that the non-square case of the Theorem \ref{silverman} also follows from the main result of \cite{bier}.

In this paper, we will prove the following:
\begin{maintheorem}
A blocking set of an affine plane of order $q$, $q\geq 25$, contains at least $q+\left\lfloor\sqrt{q}\right\rfloor+3$ points.
\end{maintheorem}

\section{The proof of the main theorem}

The following lemma will be used in the proof of our main theorem.

\begin{lemma}\label{afschatting}
	Let $b>1$ be an integer and let $k,m\in\N$ be such that $k\equiv m\pmod{b-1}$ and $m<b-1$. If $a_{1},\dots,a_{b}\in\N$ satisfy $\sum_{i=1}^{b}(i-1)a_{i}=k$, then $\sum_{i=1}^{b}(b-i)(i-1)a_{i}\geq m(b-1-m)$.
\end{lemma}
\begin{proof}
	Let $\ell$ be an integer such that $k=\ell(b-1)+m$. We first argue that $\sum_{i=1}^{b}(i-1)^{2}a_{i}\leq\ell(b-1)^{2}+m^{2}$ and that equality is reached if and only if $a_{b}=\ell$, $a_{m+1}=1$ and all $a_{i}$'s with $i\in\{2,\dots,b-1\}\setminus\{m+1\}$ equal zero. If $a_{1},\dots,a_{b}\in\N$ satisfy $\sum_{i=1}^{b}(i-1)a_{i}=k$ and we can find an $a_{j}\geq2$ for some $2\leq j\leq b-1$, then set
	\begin{align*}
	a'_{i}=\begin{cases}
	a_{i}-2 &\text{if }i=j\\
	a_{i}+1 &\text{if }\begin{cases}
	i=2j-1 &\text{if }2j\leq b\\
	i\in\{2j-b,b\}&\text{if }2j> b
	\end{cases}\\
	a_{i} &\text{else}
	\end{cases}\;.
	\end{align*}
	We find that $\sum_{i=1}^{b}(i-1)a'_{i}=k$ and that $\sum_{i=1}^{b}(i-1)^{2}a_{i}< \sum_{i=1}^{b}(i-1)^{2}a'_{i}$. If $a_{1},\dots,a_{b}\in\N$ satisfy $\sum_{i=1}^{b}(i-1)a_{i}=k$ and we can find $a_{j},a_{j'}\geq1$ for $2\leq j\neq j'\leq b-1$, then set
	\begin{align*}
	a'_{i}=\begin{cases}
	a_{i}-1 &\text{if }i\in\{j,j'\}\\
	a_{i}+1 &\text{if }\begin{cases}
	i=j+j'-1 &\text{if }j+j'\leq b\\
	i\in\{j+j'-b,b\}&\text{if }j+j'> b
	\end{cases}\\
	a_{i} &\text{else}
	\end{cases}\;.
	\end{align*}
	We find that $\sum_{i=1}^{b}(i-1)a'_{i}=k$ and that $\sum_{i=1}^{b}(i-1)^{2}a_{i}< \sum_{i=1}^{b}(i-1)^{2}a'_{i}$. From these arguments the claim follows. Now, we find immediately that
	\begin{align*}
	\sum_{i=1}^{b}(b-i)(i-1)a_{i}&=(b-1)\sum_{i=1}^{b}(i-1)a_{i}-\sum_{i=1}^{b}(i-1)^{2}a_{i}\\
	&\geq(b-1)k-\left(\ell(b-1)^{2}+m^{2}\right)=m(b-1-m)\;.\qedhere
	\end{align*}
\end{proof}

We will prove our main theorem in the dual setting. A line $\ell$ of a plane $\Pi$ is said to {\em cover} a point $P$ if $P$ lies on $\ell$. It is now easy to see that the following lemma holds.

\begin{lemma}
	The dual of an affine blocking set (considered as a subset of the projective plane $\Pi$ obtained through completion), is a set of lines in the dual of $\Pi$, covering all the points but one.
\end{lemma}

\begin{definition}
	Let $\L$ be a set of lines of a projective plane $\mathcal{P}$. A {\em $k$-knot} is a point of $\mathcal{P}$ that lies on exactly $k$ lines of $\L$.
\end{definition}

\begin{lemma}\label{Beq}
	Let $P$ be a point of an axiomatic projective plane $\mathcal{P}$ of order $q$, and let $\mathcal{L}$ be a set of lines in $\mathcal{P}$ such that all points but $P$ are on a line of $\mathcal{L}$, and $P$ is not.  Let $k$ be such that $\mathcal{L}$ has no $k'$-knot for $k'>k$. Let $x_i$ denote the number of $i$-knots in $\mathcal{P}$. Then, for all $\overline{k}$:
	\begin{align*}
	\sum_{i=1}^{k}(i-1)(\overline{k}-i)x_{i}
	=-|\mathcal{L}|(|\mathcal{L}|-1)+\overline{k}|\mathcal{L}|(q+1)-\overline{k}(q^{2}+q)\;.
	\end{align*}
\end{lemma}

\begin{proof}
	\par We denote the number of $i$-knots (points on exactly $i$ lines of $\mathcal{L}$) by $x_{i}$, $i=1,\dots,k$. We do the standard countings: first we count the number of points, then the number of pairs $\{(P,\ell)\mid P\ \mbox{is a point on the line }\ell\in\L\}$ and finally the number of triples $\{(P,\ell,m)\mid \ell,m\in\L,\ell\neq m,\ P\mbox{ a point on }\ell\mbox{ and }m\}$. We find
	\[
	\sum_{i=1}^{k}x_{i}=q^{2}+q\;,\quad\sum_{i=1}^{k}ix_{i}=|\mathcal{L}|(q+1)\;,\quad\sum_{i=1}^{k}i(i-1)x_{i}=|\mathcal{L}|(|\mathcal{L}|-1)\;.
	\]
	Now we can execute the following calculation:
	\begin{align*}
	\sum_{i=1}^{k}(i-1)(\overline{k}-i)x_{i}&=-\sum_{i=1}^{k}i(i-1)x_{i}+\overline{k}\sum_{i=1}^{k}ix_{i}-\overline{k}\sum_{i=1}^{k}x_{i}\nonumber\\
	&=-|\mathcal{L}|(|\mathcal{L}|-1)+\overline{k}|\mathcal{L}|(q+1)-\overline{k}(q^{2}+q)\;.
	\qedhere
	\end{align*}
\end{proof}

\begin{lemma}\label{eerste}
	Let $P$ be a point of an axiomatic projective plane $\mathcal{P}$ of order $q$, $q\geq9$, and let $\mathcal{L}$ be a set of less than $2q-1$ lines in $\mathcal{P}$ such that all points but $P$ are on a line of $\mathcal{L}$, and $P$ is not.  Let $k$ be such that $\mathcal{L}$ admits a $k$-knot but no $k'$-knot for $k'>k$. Then $k<q$, $|\L|\geq q+k$ and $k>\left\lfloor\sqrt{q}\right\rfloor$.
\end{lemma}
\begin{proof}
	It is immediate that $\mathcal{L}$ does not admit a $(q+1)$-knot as the $q+1$ lines through some fixed point cover all points. If $\mathcal{L}$ admits a $q$-knot $R$, then $\mathcal{L}$ has to contain at least $q-1$ lines not through $R$, in order to cover the points of $PR$ different from $P$ and $R$. So, $|\mathcal{L}|\geq 2q-1$, a contradiction.
	\par Let $k$ be such that $\mathcal{L}$ admits a $k$-knot, but no $k'$-knot for $k'>k$, and let $K$ be a $k$-knot. We may assume that $k\leq q-1$, and so we can find a line $\ell$ through $K$ not in $\mathcal{L}$ and not through $P$. Since all points on $\ell$ different from $K$ are contained in a line of $\mathcal{L}$ and these lines are necessarily different, we have $|\mathcal{L}|\geq q+k$.
	\par If $k\leq\left\lfloor\sqrt{q}\right\rfloor$, then it follows from Lemma \ref{Beq} that
	\[
		0\leq\sum_{i=1}^{k}(i-1)(k-i)x_{i}=-|\mathcal{L}|(|\mathcal{L}|-1)+k(q+1)\left(|\mathcal{L}|-q\right)\leq-|\mathcal{L}|(|\mathcal{L}|-1)+\left\lfloor\sqrt{q}\right\rfloor(q+1)\left(|\mathcal{L}|-q\right)
	\]
	where we used that $|\mathcal{L}|-q\geq k>0$. We set $|\mathcal{L}|=q+\left\lfloor\sqrt{q}\right\rfloor+m$ and we find that
	\begin{align*}
		0&\leq-(q^{2}-q\left\lfloor\sqrt{q}\right\rfloor^{2})-m^{2}-2q\left\lfloor\sqrt{q}\right\rfloor-2mq-m\left\lfloor\sqrt{q}\right\rfloor+q+\left\lfloor\sqrt{q}\right\rfloor+m+mq\left\lfloor\sqrt{q}\right\rfloor\\
		&\leq(m-2)(q\left\lfloor\sqrt{q}\right\rfloor-2q-\left\lfloor\sqrt{q}\right\rfloor-m-1)-3q-\left\lfloor\sqrt{q}\right\rfloor-2\\
		&=-(m-2)^{2}+(m-2)(q\left\lfloor\sqrt{q}\right\rfloor-2q-\left\lfloor\sqrt{q}\right\rfloor-3)-3q-\left\lfloor\sqrt{q}\right\rfloor-2\\
		&\leq (m-2)(q\left\lfloor\sqrt{q}\right\rfloor-2q-\left\lfloor\sqrt{q}\right\rfloor-3)-3q-\left\lfloor\sqrt{q}\right\rfloor-2\;.
	\end{align*}
	We know that $\left\lfloor\sqrt{q}\right\rfloor>\sqrt{q}-1$ and hence $q\left\lfloor\sqrt{q}\right\rfloor-2q-\left\lfloor\sqrt{q}\right\rfloor-3> q\sqrt{q}-3q-\sqrt{q}-2$. We can see that $q\sqrt{q}-3q-\sqrt{q}-2>0$ for $q\geq12$ and hence that $q\left\lfloor\sqrt{q}\right\rfloor-2q-\left\lfloor\sqrt{q}\right\rfloor-3>0$ for $q\geq9$. Since $q\left\lfloor\sqrt{q}\right\rfloor-2q-\left\lfloor\sqrt{q}\right\rfloor-3>0$ for $q\geq9$ and $-3q-\left\lfloor\sqrt{q}\right\rfloor-2<0$, the previous inequality gives a contradiction for $m\leq 2$.
\end{proof}

\begin{lemma}\label{plus1}
	Let $P$ be a point of an axiomatic projective plane $\mathcal{P}$ of order $q$, $q\geq 25$, and let $\mathcal{L}$ be a set of at most $q+\left\lfloor \sqrt{q}\right\rfloor+3$ lines in $\mathcal{P}$ such that all points but $P$ are on a line of $\mathcal{L}$, and $P$ is not.  Let $k$ be such that $\mathcal{L}$ admits a $k$-knot but no $k'$-knot for $k'>k$. Then $k>\left\lfloor\sqrt{q}\right\rfloor+1$.

\end{lemma}
\begin{proof}
	Note that $2q-1>q+\sqrt{q}+3\geq q+\left\lfloor\sqrt{q}\right\rfloor+3$ since $q\geq7$. From Lemma \ref{eerste}, we know that $k>\left\lfloor\sqrt{q}\right\rfloor$. If $k=\left\lfloor\sqrt{q}\right\rfloor+1$, then we have $|\mathcal{L}|\geq q+\left\lfloor\sqrt{q}\right\rfloor+1$. Setting $|\mathcal{L}|=q+\left\lfloor\sqrt{q}\right\rfloor+1+m$ it follows from Lemma \ref{Beq} that
	\begin{align}\label{Beq0}
		\sum_{i=1}^{\left\lfloor\sqrt{q}\right\rfloor+1}(i-1)(\left\lfloor\sqrt{q}\right\rfloor+1-i)x_{i}=-q\left(q-\left\lfloor\sqrt{q}\right\rfloor^{2}\right)+\left\lfloor\sqrt{q}\right\rfloor+1+m\left((q-1)\left\lfloor\sqrt{q}\right\rfloor-q-m\right)\;.
	\end{align}
	\par We need to show that for $m=0$ and $m=1$, we find a contradiction.
	\paragraph*{Case 1: $m=0$.}\mbox{}\\
	\indent If $m=0$ the total contribution of the $i$-knots, $i>0$, in the left hand side of \eqref{Beq0} equals $\left\lfloor\sqrt{q}\right\rfloor+1-q\left(q-\left\lfloor\sqrt{q}\right\rfloor^{2}\right)$. We know that 1-knots and $(\left\lfloor\sqrt{q}\right\rfloor+1)$-knots contribute zero to this sum, 2-knots and $\left\lfloor\sqrt{q}\right\rfloor$-knots contribute $\sqrt{q}-1$ to this sum, and all other knots contribute at least $2\left\lfloor\sqrt{q}\right\rfloor-4$ to this sum. Since $\left\lfloor\sqrt{q}\right\rfloor\geq5$ we have $0<\left\lfloor\sqrt{q}\right\rfloor-1<2\left\lfloor\sqrt{q}\right\rfloor-4<2\left(\left\lfloor\sqrt{q}\right\rfloor-1\right)$. We look at the different cases.
	\begin{itemize}
		\item If $\left\lfloor\sqrt{q}\right\rfloor+1-q\left(q-\left\lfloor\sqrt{q}\right\rfloor^{2}\right)=0$, then either $\left\lfloor\sqrt{q}\right\rfloor+1=0$ or $\left\lfloor\sqrt{q}\right\rfloor+1\geq q$, which both lead to a contradiction.
		\item If $\left\lfloor\sqrt{q}\right\rfloor+1-q\left(q-\left\lfloor\sqrt{q}\right\rfloor^{2}\right)=\left\lfloor\sqrt{q}\right\rfloor-1$, then $q\left(q-\left\lfloor\sqrt{q}\right\rfloor^{2}\right)=2$, which also leads to contradiction since $q\geq3$.
		\item If $\left\lfloor\sqrt{q}\right\rfloor+1-q\left(q-\left\lfloor\sqrt{q}\right\rfloor^{2}\right)\geq2\left\lfloor\sqrt{q}\right\rfloor-4$, then $q\left(q-\left\lfloor\sqrt{q}\right\rfloor^{2}\right)+\left(\left\lfloor\sqrt{q}\right\rfloor-5\right)\leq 0$. Since both terms are non-negative, it follows that $q=25$ and there is equality in the assumption. So, in this case we find that there is exactly one 3-knot or exactly one 4-knot, and all other points (different from $P$) are 1-knots or 6-knots. Say there is a 3-knot $P'$ (the argument for a 4-knot is analogous). We count the tuples $(Q,\ell)$ with $\ell\in\mathcal{L}$ and $Q$ the intersection point of $\ell$ and $PP'$. Reducing modulo 5 we find $1\equiv|\mathcal{L}|\equiv(25-1)\cdot1+3\equiv 2\pmod{5}$, a contradiction.
	\end{itemize}
	\paragraph*{Case 2: $m=1$.}\mbox{}\\
	\indent If $m=1$ the total contribution of the $i$-knots, $i>0$, in the left hand side of \eqref{Beq0} equals $q\left(\left\lfloor\sqrt{q}\right\rfloor-1\right)-q\left(q-\left\lfloor\sqrt{q}\right\rfloor^{2}\right)$. We know that $\L$ admits a $(\left\lfloor\sqrt{q}\right\rfloor+1)$-knot, say $K$. Let $\ell$ be a line through $K$ and denote the number of $i$-knots on $\ell$ by $a_{i}$.
	\par If $\ell\notin\mathcal{L}$, then we have $\sum_{i=1}^{\sqrt{q}+1}ia_{i}=q+\left\lfloor\sqrt{q}\right\rfloor+2$ by counting the tuples $(Q,n)$ with $Q$ on $n$ and $n\in\mathcal{L}$. We also have $\sum_{i=1}^{\left\lfloor\sqrt{q}\right\rfloor+1}a_{i}=q+1-\delta$ where $\delta=0$ if $P\notin\ell$ and $\delta=1$ if $P\in\ell$. We find  $\sum_{i=1}^{\left\lfloor\sqrt{q}\right\rfloor+1}(i-1)a_{i}=\left\lfloor\sqrt{q}\right\rfloor+1+\delta$. Using Lemma \ref{afschatting} with $b=\left\lfloor\sqrt{q}\right\rfloor+1$, $k=\left\lfloor\sqrt{q}\right\rfloor+\delta+1$, $m=\delta+1$, we find that the points on $\ell$ contribute at least $\left\lfloor\sqrt{q}\right\rfloor-1$ to the sum in \eqref{Beq0} if $P\notin\ell$ (so if $\delta=0$) and that the points on the line $KP$ (so, for $\delta=1$) contribute at least $2(\left\lfloor\sqrt{q}\right\rfloor-2)$.
	\par If $\ell\in\mathcal{L}$ we have $\sum_{i=1}^{\left\lfloor\sqrt{q}\right\rfloor+1}(i-1)a_{i}=q+\left\lfloor\sqrt{q}\right\rfloor+1$ by counting the tuples $(Q,n)$ with $Q$ on $n$ and $n\in\mathcal{L}\setminus\{\ell\}$. We now distinguish between two cases.
	\begin{itemize}
		\item[(i)] If $q+\left\lfloor\sqrt{q}\right\rfloor+1\not\equiv0\pmod{\left\lfloor\sqrt{q}\right\rfloor}$, then we find that not all points on $\ell\in\mathcal{L}$ can be 1-knots and $\left(\left\lfloor\sqrt{q}\right\rfloor+1\right)$-knots, hence the points on $\ell$ contribute at least $\left\lfloor\sqrt{q}\right\rfloor-1$ to the sum in \eqref{Beq0}. It follows that
		\[
			q\left(\left\lfloor\sqrt{q}\right\rfloor-1\right)-q\left(q-\left\lfloor\sqrt{q}\right\rfloor^{2}\right)=\sum_{i=1}^{\sqrt{q}+1}(i-1)(\left\lfloor\sqrt{q}\right\rfloor+1-i)x_{i}\geq q(\left\lfloor\sqrt{q}\right\rfloor-1)+2(\left\lfloor\sqrt{q}\right\rfloor-2)\;,
		\]
		hence
		\[
			q\left(q-\left\lfloor\sqrt{q}\right\rfloor^{2}\right)+2\left(\left\lfloor\sqrt{q}\right\rfloor-2\right)\leq 0\;,
		\]
		a contradiction since $q>4$. Note that the contribution of $K$ is necessarily 0 so it does not matter that we counted this point $q+1$ times.
		\item[(ii)] If $q+\left\lfloor\sqrt{q}\right\rfloor+1\equiv0\pmod{\left\lfloor\sqrt{q}\right\rfloor}$, then $q=\left\lfloor\sqrt{q}\right\rfloor\left(\left\lfloor\sqrt{q}\right\rfloor+\varepsilon\right)-1$ for some $\varepsilon\in\{1,2\}$ since $\left\lfloor\sqrt{q}\right\rfloor\left(\left\lfloor\sqrt{q}\right\rfloor+0\right)-1\leq\sqrt{q}\sqrt{q}-1<q$ and $\left\lfloor\sqrt{q}\right\rfloor\left(\left\lfloor\sqrt{q}\right\rfloor+3\right)-1>(\sqrt{q}-1)(\sqrt{q}+2)-1\geq q$ for $q\geq9$. Using the previous argument, only looking at lines through $K$ not in $ \L$, we then find
		\begin{align*}
			q\left(\left\lfloor\sqrt{q}\right\rfloor-1\right)-q\left(q-\left\lfloor\sqrt{q}\right\rfloor^{2}\right)&=\sum_{i=1}^{\sqrt{q}+1}(i-1)(\left\lfloor\sqrt{q}\right\rfloor+1-i)x_{i}\\
			&\geq (q-\left\lfloor\sqrt{q}\right\rfloor-1)(\left\lfloor\sqrt{q}\right\rfloor-1)+2(\left\lfloor\sqrt{q}\right\rfloor-2)\;,
		\end{align*}
		hence
		\begin{align*}
			&q\left(q-\left\lfloor\sqrt{q}\right\rfloor^{2}\right)-\left(\left\lfloor\sqrt{q}\right\rfloor^{2}-2\left\lfloor\sqrt{q}\right\rfloor+3\right)\leq 0\;,\\
			\Leftrightarrow\quad&\varepsilon\left\lfloor\sqrt{q}\right\rfloor^{3}+(\varepsilon^{2}-2)\left\lfloor\sqrt{q}\right\rfloor^{2}-2(\varepsilon-1)\left\lfloor\sqrt{q}\right\rfloor-2\leq 0
		\end{align*}
		a contradiction for $\varepsilon\in\{1,2\}$ since $q\geq4$.\qedhere
	\end{itemize}
\end{proof}

\begin{lemma}\label{5superknots}
	Let $P$ be a point of an axiomatic projective plane $\mathcal{P}$ of order $q$, $q\geq 25$, and let $\mathcal{L}$ be a set of lines in $\mathcal{P}$ such that all points but $P$ are on a line of $\mathcal{L}$, and $P$ is not. Let $k$ be such that $\mathcal{L}$ admits a $k$-knot but no $k'$-knot for $k'>k$. Suppose that $k=\lfloor \sqrt{q}\rfloor+2$ and that $\mathcal{L}$ has size $q+\lfloor \sqrt{q}\rfloor+2$, then the number of $(\lfloor \sqrt{q}\rfloor+2)$-knots is at most $5$.
\end{lemma}
\begin{proof}
	We will first prove that there can be at most three $\left(\left\lfloor\sqrt{q}\right\rfloor+2\right)$-knots on a line. Assume that $K_{1},K_{2},K_{3},K_{4}$ are four $\left(\left\lfloor\sqrt{q}\right\rfloor+2\right)$-knots on a line; necessarily this line is in $\mathcal{L}$. The point $P$ is not on $\ell$. On the line $PK_{1}$ there is precisely on 2-knot $T$; the remaining $q-1$ points are 1-knots. Consider the lines $TK_{i}$, $i=2,3,4$. They do not contain the point $P$, and hence, if $TK_i$ is not a line of $\L$, all $q-1$ points, different from $T$ and $K_i$ have to be covered by one of the remaining $q+\left\lfloor\sqrt{q}\right\rfloor+2-(\left\lfloor\sqrt{q}\right\rfloor+2)-2=q-2$ lines of $\L$ not through $T$ and $K_i$ which is impossible. We conclude that the lines $TK_i$, $i=2,3,4$ are lines of $\L$. But this is a contradiction as $T$ is on exactly two lines of $\mathcal{L}$.
	\par We now prove that $d\leq 5$. Let $K_{1},\dots,K_{6}$ be six $\left(\left\lfloor\sqrt{q}\right\rfloor+2\right)$-knots. For any $K_{i}$, $i=1,\dots,6$, we know that there is exactly one 2-knot $T_{i}$ on the line $PK_{i}$. If $i\neq j$, then $T_{i}K_{j}$ is a line not containing $P$, hence, using the same argument as above, $T_iK_j$ is a line of $\mathcal{L}$. As there are only two lines of $\mathcal{L}$ through $T_{i}$, $i=1,\dots,6$, one of them contains at least three $\left(\left\lfloor\sqrt{q}\right\rfloor+2\right)$-knots. So, we know that three of the points $K_{1},\dots,K_{6}$ are collinear, say $K_{1}$, $K_{2}$ and $K_{3}$, and we denote the line they are on by $\ell$. As there are only two lines of $\mathcal{L}$ through a 2-knot we have immediately that $T_{i}\in\ell$ for $i=4,5,6$. Moreover, the other line of $\mathcal{L}$ through $T_{i}$ contains $K_{j}$ and $K_{j'}$ for $\{i,j,j'\}=\{4,5,6\}$. Now, the two lines of $\mathcal{L}$ through $T_{1}$ must be $T_{1}K_{2}$ and $T_{1}K_{3}$. Each of the $\left(\left\lfloor\sqrt{q}\right\rfloor+2\right)$-knots $K_{4}$, $K_{5}$ and $K_{6}$ must be on one of these two lines, so one of these two lines contains at least two of them. Say, without loss of generality that $K_{4}$ and $K_{5}$ are on $T_{1}K_{2}$. However, then $T_{6}$ is also on this line, which implies $T_{6}=K_{2}$ since both points are also on $\ell$, a contradiction.
\end{proof}

\begin{theorem}
	Let $P$ be a point of an axiomatic projective plane $\mathcal{P}$ of order $q$, and let $\mathcal{L}$ be a set of lines in $\mathcal{P}$ such that all points but $P$ are on a line of $\mathcal{L}$, and $P$ is not. If $q\geq25$, then $|\mathcal{L}|\geq q+\left\lfloor\sqrt{q}\right\rfloor+3$.
\end{theorem}
\begin{proof}
	By Lemma \ref{eerste}, we have that $|\mathcal{L}|\geq q+\left\lfloor\sqrt{q}\right\rfloor+3$ if $k\geq\left\lfloor\sqrt{q}\right\rfloor+3$. By Lemma \ref{plus1}, $k>\left\lfloor\sqrt{q}\right\rfloor+1$. So it follows that we only need to rule out the case $k=\left\lfloor\sqrt{q}\right\rfloor+2$ in order to obtain that $|\mathcal{L}|\geq q+\left\lfloor\sqrt{q}\right\rfloor+3$.
	\par If $k=\left\lfloor\sqrt{q}\right\rfloor+2$, then we know $|\mathcal{L}|\geq q+\left\lfloor\sqrt{q}\right\rfloor+2$ and we may assume that $|\mathcal{L}|=q+\left\lfloor\sqrt{q}\right\rfloor+2$. We set $d=x_{\left\lfloor\sqrt{q}\right\rfloor+2}$ for ease of notation. It then follows from Lemma \ref{Beq} that
	\begin{align}\label{Beq1}
		\sum_{j=1}^{\left\lfloor\sqrt{q}\right\rfloor+1}(j-1)(\left\lfloor\sqrt{q}\right\rfloor+1-j)x_{j}=q\left(\left\lfloor\sqrt{q}\right\rfloor-1\right)-q\left(q-\left\lfloor\sqrt{q}\right\rfloor^{2}\right)+d(\left\lfloor\sqrt{q}\right\rfloor+1)\;.
	\end{align}
	Now, let $K$ be an $i$-knot with $4\leq i\leq\left\lfloor\sqrt{q}\right\rfloor$, let $\ell\notin\mathcal{L}$ be a line through $K$ and denote the number of $j$-knots on $\ell$ different from $K$ and $P$ by $a_{j}$. Then we have $\sum_{j=1}^{\left\lfloor\sqrt{q}\right\rfloor+2}ja_{j}=q+\left\lfloor\sqrt{q}\right\rfloor+2-i$ by counting the tuples $(Q,n)$ with $\{Q\}=\ell\cap n$, $Q\neq K$ and $n\in\mathcal{L}$. We also have $\sum_{j=1}^{\left\lfloor\sqrt{q}\right\rfloor+2}a_{j}=q-\delta$ where $\delta=0$ if $P\notin\ell$ and $\delta=1$ if $P\in\ell$. We find  $\sum_{j=1}^{\left\lfloor\sqrt{q}\right\rfloor+2}(j-1)a_{j}=\left\lfloor\sqrt{q}\right\rfloor+2-i+\delta$. Moreover, as $\left\lfloor\sqrt{q}\right\rfloor+2-i+\delta\leq\left\lfloor\sqrt{q}\right\rfloor-1$, we have $\sum_{j=1}^{\left\lfloor\sqrt{q}\right\rfloor+1}(j-1)a_{j}=\left\lfloor\sqrt{q}\right\rfloor+2-i+\delta$. So we find that the points on $\ell\neq KP$ different from $K$ contribute at least $\left(\left\lfloor\sqrt{q}\right\rfloor+2-i\right)(i-2)$ to the sum in \eqref{Beq1} by Lemma \ref{afschatting}; the points on $KP$ different from $K$ contribute at least $\left(\left\lfloor\sqrt{q}\right\rfloor+3-i\right)(i-3)$ to the sum in \eqref{Beq1}. None of these contributions includes the point $K$, but $K$ itself contributes $\left(\left\lfloor\sqrt{q}\right\rfloor+1-i\right)(i-1)$ to the sum in \eqref{Beq1}. We find that
	\begin{align}\label{Beq2pro}
		q\left(\left\lfloor\sqrt{q}\right\rfloor-1\right)-q\left(q-\left\lfloor\sqrt{q}\right\rfloor^{2}\right)+d(\left\lfloor\sqrt{q}\right\rfloor+1)&=\sum_{j=1}^{\left\lfloor\sqrt{q}\right\rfloor+1}(j-1)\left(\left\lfloor\sqrt{q}\right\rfloor+1-j\right)x_{j}\nonumber\\
		&\geq (q-i)\left(\left\lfloor\sqrt{q}\right\rfloor+2-i\right)(i-2)\nonumber\\&\ \quad+\left(\left\lfloor\sqrt{q}\right\rfloor+3-i\right)(i-3)+\left(\left\lfloor\sqrt{q}\right\rfloor+1-i\right)(i-1)\;.
	\end{align}
	Hence,
	\begin{align}\label{Beq2}
		0&\geq i^{3}-(q+\left\lfloor\sqrt{q}\right\rfloor+6)i^{2}+(q\left\lfloor\sqrt{q}\right\rfloor+4q+4\left\lfloor\sqrt{q}\right\rfloor+12)i-d(\left\lfloor\sqrt{q}\right\rfloor+1)\nonumber\\&\qquad+q^{2}-q\left\lfloor\sqrt{q}\right\rfloor^{2}-3q\left\lfloor\sqrt{q}\right\rfloor-3q-4\left\lfloor\sqrt{q}\right\rfloor-10\;.
	\end{align}
	For $i=4$ we find, using that $d\leq 5$ from Lemma \ref{5superknots},
	\[
		0\geq q^{2}-q\left\lfloor\sqrt{q}\right\rfloor^{2}+q\left\lfloor\sqrt{q}\right\rfloor-3q -4\left\lfloor\sqrt{q}\right\rfloor+6-d\left(\left\lfloor\sqrt{q}\right\rfloor+1\right)\geq q\left\lfloor\sqrt{q}\right\rfloor-3q -9\left\lfloor\sqrt{q}\right\rfloor+1\;,
	\]
	which is a contradiction for $q\geq25$. For $i=\left\lfloor\sqrt{q}\right\rfloor$ we find, again using that $d\leq 5$ from Lemma \ref{5superknots},
	\begin{align*}
		0&\geq q^{2}-q\left\lfloor\sqrt{q}\right\rfloor^{2}+q\left\lfloor\sqrt{q}\right\rfloor-3q-2\left\lfloor\sqrt{q}\right\rfloor^{2}+8\left\lfloor\sqrt{q}\right\rfloor-10-d\left(\left\lfloor\sqrt{q}\right\rfloor+1\right)\\
		&\geq q\left\lfloor\sqrt{q}\right\rfloor-3q-2\left\lfloor\sqrt{q}\right\rfloor^{2}+3\left\lfloor\sqrt{q}\right\rfloor-15\;,
	\end{align*}
	which is a contradiction for $q>25$. Since the right hand side of \eqref{Beq2} is function of degree 3 in $i$ it can easily be checked that it is first increasing and then decreasing on the interval $[4,\left\lfloor\sqrt{q}\right\rfloor]$ for $q\geq25$. So, the minimum of this function on this interval is to be found at one of its endpoints. So, as we have a contradiction for $i=4$ and $i=\left\lfloor\sqrt{q}\right\rfloor$, we have a contradiction for all $i$ with $4\leq i\leq\left\lfloor\sqrt{q}\right\rfloor$.
	\par For $q=25$ we did not yet exclude $i=5$. We find a contradiction in \eqref{Beq2} if $0\leq d\leq 4$, and we have equality in \eqref{Beq2} if $d=5$. In the latter case we also have equality in \eqref{Beq2pro} and hence the points on the lines of $\mathcal{L}$ through the $5$-knot $K$ should all contribute zero to the sum in \eqref{Beq1}. If $\ell\in\mathcal{L}$ through $K$ we have
	\[
		\sum_{j=1}^{7}(j-1)a_{j}=q+\left\lfloor\sqrt{q}\right\rfloor+2-i=27
	\]
	by counting the tuples $(Q,n)$ with $\{Q\}=\ell\cap n$, $Q\neq K$ and $n\in\mathcal{L}\setminus\{\ell\}$. Hence, not all points on $\ell$ different from $K$ can be 1-knots or 6-knots; more precisely, there should be at least two 7-knots on $\ell$. It follows that the total number of 7-knots is at least 10, a contradiction.
	\par So far we have concluded that $x_{i}=0$ for all $i$ with $4\leq i\leq\left\lfloor\sqrt{q}\right\rfloor$. Now, let $K'$ be a 3-knot, let $\ell\notin\mathcal{L}$ be a line through $K'$ different from $K'P$ and denote the number of $j$-knots on $\ell$ different from $K'$ by $a_{j}$. Counting the tuples $(Q,n)$ with $\{Q\}=\ell\cap n$, $Q\neq K'$ and $n\in\mathcal{L}$ and using $\sum_{j=1}^{\left\lfloor\sqrt{q}\right\rfloor+2}a_{j}=q$ we find $\sum_{j=1}^{\left\lfloor\sqrt{q}\right\rfloor+2}(j-1)a_{j}=\left\lfloor\sqrt{q}\right\rfloor-1$, but as $a_{j}=0$ for all $j=4,\dots,\left\lfloor\sqrt{q}\right\rfloor$, we find that
	\[
		a_{2}+2a_{3}+\left\lfloor\sqrt{q}\right\rfloor a_{\left\lfloor\sqrt{q}\right\rfloor+1}+\left(\left\lfloor\sqrt{q}\right\rfloor+1\right)a_{\left\lfloor\sqrt{q}\right\rfloor+2}=\left\lfloor\sqrt{q}\right\rfloor-1
	\]
	and hence that $a_{\left\lfloor\sqrt{q}\right\rfloor+1}=a_{\left\lfloor\sqrt{q}\right\rfloor+2}=0$. So, the points on $\ell$ different from $K$ contribute
	\begin{align*}
		a_{2}\left(\left\lfloor\sqrt{q}\right\rfloor-1\right)+a_{3}\left(2\left(\left\lfloor\sqrt{q}\right\rfloor-2\right)\right)&=a_{2}\left(\left\lfloor\sqrt{q}\right\rfloor-1\right)+\frac{\left\lfloor\sqrt{q}\right\rfloor-1-a_{2}}{2}\left(2\left(\left\lfloor\sqrt{q}\right\rfloor-2\right)\right)\\
		&=a_{2}+\left(\left\lfloor\sqrt{q}\right\rfloor-1\right)\left(\left\lfloor\sqrt{q}\right\rfloor-2\right)
	\end{align*}
	hence at least $\left(\left\lfloor\sqrt{q}\right\rfloor-1\right)\left(\left\lfloor\sqrt{q}\right\rfloor-2\right)$ to the sum in \eqref{Beq1}. The point $K$ itself also contributes $2\left(\left\lfloor\sqrt{q}\right\rfloor-2\right)$ to this sum. We find 
	\begin{align*}
		q\left(\left\lfloor\sqrt{q}\right\rfloor-1\right)-q\left(q-\left\lfloor\sqrt{q}\right\rfloor^{2}\right)+d\left(\left\lfloor\sqrt{q}\right\rfloor+1\right)&=\sum_{j=1}^{\left\lfloor\sqrt{q}\right\rfloor+1}(j-1)\left(\left\lfloor\sqrt{q}\right\rfloor+1-j\right)x_{j}\\
		&\geq (q-3)\left(\left\lfloor\sqrt{q}\right\rfloor-1\right)\left(\left\lfloor\sqrt{q}\right\rfloor-2\right)+2\left(\left\lfloor\sqrt{q}\right\rfloor-2\right)\\
		&=q\left\lfloor\sqrt{q}\right\rfloor^{2}-3q\left\lfloor\sqrt{q}\right\rfloor+2q-3\left\lfloor\sqrt{q}\right\rfloor^{2}\\&\qquad+11\left\lfloor\sqrt{q}\right\rfloor-10\;.
	\end{align*}
	So, we have, using Lemma \ref{5superknots}, that
	\[
		0\geq q^{2}-4q\left\lfloor\sqrt{q}\right\rfloor+3q-3\left\lfloor\sqrt{q}\right\rfloor^{2}+11\left\lfloor\sqrt{q}\right\rfloor-10-d\left(\left\lfloor\sqrt{q}\right\rfloor+1\right)\geq q^{2}-4q\left\lfloor\sqrt{q}\right\rfloor+6\left\lfloor\sqrt{q}\right\rfloor-15\;,
	\]
	which is a contradiction for $q\geq16$.
	\par So far, we know that $x_{i}=0$ for all $i\notin\{1,2,\left\lfloor\sqrt{q}\right\rfloor+1,\left\lfloor\sqrt{q}\right\rfloor+2\}$. We call the $\left(\left\lfloor\sqrt{q}\right\rfloor+1\right)$-knots and the $\left(\left\lfloor\sqrt{q}\right\rfloor+2\right)$-knots {\em big knots}. It follows from the standard countings that
	\begin{align*}
		&\left\{\begin{aligned}
		x_{1}+x_{2}+x_{\left\lfloor\sqrt{q}\right\rfloor+1}&=q^{2}+q-d\\
		x_{1}+2x_{2}+\left(\left\lfloor\sqrt{q}\right\rfloor+1\right)x_{\left\lfloor\sqrt{q}\right\rfloor+1}&=(q+1)\left(q+\left\lfloor\sqrt{q}\right\rfloor+2\right)-\left(\left\lfloor\sqrt{q}\right\rfloor+2\right)d\\
		2x_{2}+\left(\left\lfloor\sqrt{q}\right\rfloor+1\right)\left\lfloor\sqrt{q}\right\rfloor x_{\left\lfloor\sqrt{q}\right\rfloor+1}&=\left(q+\left\lfloor\sqrt{q}\right\rfloor+2\right)\left(q+\left\lfloor\sqrt{q}\right\rfloor+1\right)-\left(\left\lfloor\sqrt{q}\right\rfloor+2\right)\left(\left\lfloor\sqrt{q}\right\rfloor+1\right)d
		\end{aligned}\right.\\
		\Leftrightarrow\quad&\left\{\begin{aligned}
		x_{1}&=q^{2}-q\left\lfloor\sqrt{q}\right\rfloor-q-d-1+\frac{q^{2}-q-2}{\left\lfloor\sqrt{q}\right\rfloor}\\
		x_{2}&=q\left\lfloor\sqrt{q}\right\rfloor+2q+d-\frac{q^{2}-q-2d}{\left\lfloor\sqrt{q}\right\rfloor-1}\\
		x_{\left\lfloor\sqrt{q}\right\rfloor+1}&=1-d+\frac{q^{2}-q-2(d-1)\left\lfloor\sqrt{q}\right\rfloor-2}{\left\lfloor\sqrt{q}\right\rfloor\left(\left\lfloor\sqrt{q}\right\rfloor-1\right)}
		\end{aligned}\right.
	\end{align*}
	Since $x_{1}$ is an integer, there must be an integer $a$ such that $q^{2}-q-2=a\left\lfloor\sqrt{q}\right\rfloor$. Since $x_{2}$ is an integer too,  $q^2-q-2d=a\left\lfloor\sqrt{q}\right\rfloor-2(d-1)\equiv0\pmod{\left\lfloor\sqrt{q}\right\rfloor-1}$. So, there must be an integer $b$ such that $a=2(d-1)+b(\left\lfloor\sqrt{q}\right\rfloor-1)$.
	We find that
	\begin{align*}
		\left\{\begin{aligned}
		x_{1}&=q^{2}-q\left\lfloor\sqrt{q}\right\rfloor-q+b\left\lfloor\sqrt{q}\right\rfloor+d-3-b\\
		x_{2}&=q\left\lfloor\sqrt{q}\right\rfloor+2q-d-b\left\lfloor\sqrt{q}\right\rfloor+2\\
		x_{\left\lfloor\sqrt{q}\right\rfloor+1}&=1-d+b	\end{aligned}\right.
	\end{align*}
	By counting the tuples $(Q,n)$ with $\{Q\}=\ell\cap n$, $Q\neq P$ and $n\in\mathcal{L}$ and using $\sum_{j=1}^{\left\lfloor\sqrt{q}\right\rfloor+2}a_{j}=q$ we find that for a line $\ell\ni P$ (necessarily not in $\mathcal{L}$) we have that $\sum_{j=1}^{\left\lfloor\sqrt{q}\right\rfloor+2}(j-1)a_{j}=\left\lfloor\sqrt{q}\right\rfloor+2$ with $a_{j}$ the number of $j$-knots on $\ell$. We see that there can be at most 1 big knot on a line through $P$. Hence, $d+x_{\left\lfloor\sqrt{q}\right\rfloor+1}$, the number of big knots, is at most $q+1$. It follows that $b\leq q$. Consequently,
	\[
		q^{2}-q-2=\left\lfloor\sqrt{q}\right\rfloor\left(2(d-1)+b(\left\lfloor\sqrt{q}\right\rfloor-1)\right)\leq q\left\lfloor\sqrt{q}\right\rfloor^{2}-q\left\lfloor\sqrt{q}\right\rfloor+2(d-1)\left\lfloor\sqrt{q}\right\rfloor\;.
	\]
	However, then, using that $d\leq 5$ from Lemma \ref{5superknots},
	\begin{align*}
		0\geq q^{2}-q\left\lfloor\sqrt{q}\right\rfloor^{2}+q\left\lfloor\sqrt{q}\right\rfloor-q-2(d-1)\left\lfloor\sqrt{q}\right\rfloor-2&\geq \left(q^{2}-q\left\lfloor\sqrt{q}\right\rfloor^{2}\right)+q\left(\left\lfloor\sqrt{q}\right\rfloor-1\right)-8\left\lfloor\sqrt{q}\right\rfloor-2\\
		&\geq q\left(\left\lfloor\sqrt{q}\right\rfloor-1\right)-8\left\lfloor\sqrt{q}\right\rfloor-2\;,
	\end{align*}
	which is a contradiction for $q\geq10$.		
\end{proof}

\begin{corollary}
	A blocking set of an affine plane of order $q$ contains at least $q+\left\lfloor\sqrt{q}\right\rfloor+3$ points if $q\geq25$.
\end{corollary}

\paragraph*{Acknowledgement:} This research was performed when the first author was visiting the School of Mathematics and Statistics at the University of Canterbury. He wants to thank the School, and in particular the second author, for their hospitality.

\noindent Maarten De Boeck\\
Universiteit Gent\\
Vakgroep Wiskunde\\
Krijgslaan 281--S25\\
B--9000 Gent\\
Flanders, Belgium\\
maarten.deboeck@ugent.be\\

\noindent Geertrui Van de Voorde\\
University of Canterbury\\
School of Mathematics and Statistics\\
Private Bag 4800\\
8140 Christchurch\\
New Zealand\\
geertrui.vandevoorde@canterbury.ac.nz

\end{document}